\documentclass[a4paper,11pt]{amsart}
\usepackage[utf8x]{inputenc}
\usepackage{amsthm,amsfonts,amsmath,amssymb}
\usepackage{verbatim} 
\usepackage{rotating} 

\usepackage{enumerate}

\usepackage{hyperref}
\usepackage{pdfsync}
\usepackage{cancel}

\usepackage{multirow}
\usepackage{color}

\usepackage{tikz}

\renewcommand{\L}{\mathcal{L}}

\newcommand{\N}{\mathcal{N}}

\newcommand{\K}{\mathbb{K}}

\newcommand{\C}{\mathbb{C}}

\newcommand{\g}{\mathfrak{g}}
\renewcommand{\v}{\mathfrak{v}}
\newcommand{\z}{\mathfrak{z}}
\newcommand{\n}{\mathfrak{n}}
\newcommand{\h}{\mathfrak{h}}
\renewcommand{\a}{\mathfrak{a}}

\renewcommand{\Im}{\operatorname{Im}}

\newcommand{\nil}{\operatorname{nil}}

\renewcommand{\O}{\mathcal{O}}

\newtheorem{theorem}{Theorem}[section]
\newtheorem{lemma}[theorem]{Lemma}
\newtheorem{corollary}[theorem]{Corollary}

\theoremstyle{definition}

\theoremstyle{remark}
\newtheorem{remark}[theorem]{Remark}

\newtheorem{some-results}[theorem]{Some known results}
\newtheorem{some-questions}[theorem]{Some open questions}

\numberwithin{equation}{section}

\begin{document}

	\title{On the rigidity of $k$-step nilpotent \\ graph Lie algebras}
	\author{Josefina Barrionuevo $\dag$ and Paulo Tirao $\dag$
		$\ddag$}

	\address{$\dag$ CIEM-FaMAF, CONICET-Universidad Nacional de Córdoba \\
		Ciudad Universitaria, 5000 Córdoba, Argentina}
	\address{$\ddag$ Guangdong Technion Israel Institute of Technology \\
		241 Daxue Road, Jinping District, Shantou, Guangdong Province,
		China.}
	\email{paulotirao@gmail.com}

	\date{October, 2025}

	\begin{abstract}
		We thoroughly explore the class of $k$-step nilpotent Lie algebras associated with a simple graph
		looking for $k$-step nilpotent Lie algebras which are rigid in the variety of at most $k$-step nilpotent Lie algebras.
		We find out that, besides the complete graph, the only examples arise for $k=2$ and graphs of at most 4 vertices.

		A key tool to prove non-$k$-rigidity in this context, is a general construction of non-trivial deformations for
		naturally graded nilpotent Lie algebras.
	\end{abstract}

	\maketitle
	
\setcounter{section}{-1}
\section{Introduction}
	
	Rigidity problems about Lie algebras in general, and about nilpotent Lie algebras in particular, are challenging.
	A major open problem is to understand whether there exist rigid nilpotent Lie algebras or not.
	The so-called Vergne's Conjecture predicts that the answer is no.
	
	Some intriguing questions on the nilpotent case are:
	\begin{enumerate}
		\item Are there $k$-step nilpotent Lie algebras which are rigid in the variety of at most $k$-step nilpotent Lie algebras?
		\item Are there $k$-step nilpotent Lie algebras which are rigid in the variety of at most $k+1$-step nilpotent Lie algebras?
		\item Are there nilpotent Lie algebras which are rigid in the variety of nilpotent Lie algebras?
	\end{enumerate}

	We know that the answer to the first questions is yes and we know only two families of such algebras (besides some isolated cases of dimension $<6$).
	Namely, for every $k$ the free $k$-step nilpotent Lie algebras, and Heisenberg Lie algebras for $k=2$ \cite{BTS}.
	We don't know whether to expect many or few.

	We don't know any examples answering positively questions two and three, besides the exceptional cases in dimension $<6$,
	due to the fact that there are only finitely many orbits (isomorphism classes) there.

	In light of the above questions, we explore a class of $k$-step nilpotent Lie algebras that are associated with simple graphs.
	Reasons for considering this class include the fact that they enjoy a pretty simple structure with a strong combinatorial flavor,
	making them accessible for explicit algebraic computations.
	And the fact that the free $k$-step nilpotent Lie algebras belong to this class; they are associated with complete graphs.

	We find two new examples answering positively question one.
	Both of them occur for $k=2$, and are associated with graphs of 4 vertices.
	Moreover, we prove that these are the only ones within the class considered.
	That is, $k$-step nilpotent graph Lie algebras are in general not $k$-rigid,
	with the distinguished exceptions of the free $k$-step nilpotent Lie algebras.
	Notice that the 3-dimensional Heisenberg Lie algebra is the free 2-step nilpotent Lie algebra
	on 2 generators and it is a graph Lie algebra.

	The fact that in general $k$-step graph Lie algebras are not $k$-rigid
	follows as a particular instance of a more general result that we prove.
	We show that a large class of $k$-step naturally graded nilpotent Lie algebras are not $k$-rigid
	by constructing explicitly non-trivial deformations of them.

	\medskip

	We would like to mention that 2-step graph Lie algebras, and more generally, 2-step and 3-step graded Lie algebras have been considered within the framework of
	deformations and contractions in \cite{AAA} and \cite{GR} respectively.

	\medskip

	All Lie algebras in this paper are finite dimensional over an algebraically closed field $\K$ of characteristic zero.

	\section{Preliminaries on $k$-step nilpotent graph Lie algebras}

	The central descending series of a Lie algebra $\g$ is the filtration of ideals
	$\g^1\supseteq \g^2 \supseteq \dots \supseteq \g^k \supseteq \dots$ defined recursively by:
	\[ \g^1=\g \quad\text{and}\quad \g^{i+1}=[\g,\g^i], \quad\text{for $i\ge 1$}. \]
	Note that it satisfies \([\g^i,\g^j]\subseteq \g^{i+j}\), for all \(1\leq i, j\).  
	
	The associated graded Lie algebra is
	\[  \operatorname{gr}(\g)= \bigoplus_{i\ge 1} \g^i/\g^{i+1}.  \]
	Its Lie bracket is given by
	\[ [\overline{x},\overline{y}] = \overline{[x,y]},\]
	for $x\in\g^i$ and $y\in\g^j$, where $\overline{x}$ and $\overline{y}$ denote the cosets in
	$\g^i/\g^{i+1}$ and $\g^j/\g^{j+1}$ respectively, and $\overline{[x,y]}$ is the coset of $[x,y]$ in $\g^{i+j}/\g^{i+j+1}$.
	
	The Lie algebra $\g$ is said to be $k$-step nilpotent if $\g^{k+1}=0$ and $\g^{k}\ne 0$
	and it is said to be \emph{naturally graded} if $\g\simeq \operatorname{gr}(\g)$.

	A Lie algebra $\g$ with a graded decomposition
	\begin{equation}\label{eqn:graded}
		\g=V_1\oplus\dots\oplus V_k,
	\end{equation}
	where $[V_i,V_j]\subseteq V_{i+j}$, is at most $k$-step nilpotent.
	In general, such a $\g$ may admit different graded decompositions.

	If in addition it satisfies that  
	\begin{equation}\label{eqn:carnot}
		\g^i=V_{i}\oplus\dots\oplus V_k,
	\end{equation}
	then it is $k$-step nilpotent and it is naturally graded.
	Such a decomposition is called a \emph{Carnot grading} and it is unique up to isomorphism.
	Every $\g$ with a graded decomposition as in \eqref{eqn:graded} admits a Carnot grading satisfying \eqref{eqn:carnot}.

	\begin{remark}
		Having a Carnot grading and being naturally graded are equivalent.
	\end{remark}

	The free $k$-step nilpotent Lie algebra on $m$ generators, $L_{(k)}(m)$, admits a Carnot grading
	$L_{(k)}(m)=L_1\oplus\dots\oplus L_k$, where $L_1=\langle v_1,\dots,v_m\rangle$ and $L_2=\Lambda^2 L_1$.
	See \cite{BTS} for a thorough description, including Hall bases.

	More general examples are the $k$-step nilpotent Lie algebras $\g(k,G)$ associated with a simple graph,
	considered by Dani and Maikar in \cite{DM} for $k=2$ and by Maikar in \cite{M} for arbitrary $k$.
	In this context, $L_{(k)}(m)=\g(k,K_m)$ where $K_m$ is the complete graph on $m$-vertices.

	Given a simple graph $G=(V,E)$, with vertices $V=\{v_1,\dots,v_m\}$ and edges $E\subseteq \mathcal{P}_2(V)$ (the collection of subsets of $V$ with two elements),
	consider the free $k$-step nilpotent Lie algebra generated by $V$, $L_{(k)}(m)$ and let $I$ be the ideal of $L_{(k)}(m)$ generated by
	$E'=\{ v_i\wedge v_j:\, \{v_i,v_j\}\not\in E \}\subseteq L_2$. $I$ is homogeneous and $I=I_2\oplus\dots\oplus I_k$ where $I_j=I\cap L_j\subseteq L_j$.
	Then,
	\[ \g(k,G)=L_{(k)}(m)/I, \]
	and a Carnot grading of it is given by
	\begin{eqnarray*}\label{eqn:V1k}
		\g(k,G) &=& L_1\oplus L_2/I_2\oplus\dots\oplus L_k/I_k \\
		&=& V_1\oplus V_2\oplus\dots\oplus V_k.
	\end{eqnarray*}

	If the graph $G$ has connected components $G_1,\dots,G_l$, then
	$\g(k,G)\simeq \g(k,G_1)\oplus\dots\oplus\g(k,G_l)$ and in particular,
	each isolated vertex of $G$ corresponds to a 1-dimensional abelian factor of $\g(k,G)$.

	\begin{remark}
		The infinite dimensional (non-nilpotent) analogs of these algebras were considered already in \cite{DK}.
		Given $G$ one considers the free Lie algebra on $m$ generators, $L(m)$, the ideal $I$ of it generated by $E'$
		and then the quotient
		\[
		\g(G)=L(m)/I=L_1\oplus L_2/I_2\oplus\dots\oplus L_k/I_k\oplus\dots .
		\]
		Note that
		\[ \g(k,G)\simeq \g(G)/\g(G)^{k+1}. \]
	\end{remark}

	\subsection{A Hall type basis of $\g(k,G)$}\label{subsec:hall}

	Given a simple graph $G$ let $G'$ be its complementary graph.
	In \cite[Theorem 5.12, Section 5]{W} a combinatorial algorithm with input the ordered set of vertices $v_1<v_2<\dots<v_m$ and final output
	a basis $\{\omega_1,\dots,\omega_n\}$ for $\g(k,G)$ is described.
	Starting from $G'$ it constructs recursively Lyndon elements (represented by words on the vertices) and then Lyndon brackets to end up with a basis for $\g(k,G)$.
	The algorithm provides bases for each $V_j$ (see \eqref{eqn:V1k}) were all the corresponding Lyndon elements are of length $j$.

	As an example (see \cite[Examples 5.13 and 5.16]{W}) consider the simple graph $G$ (left) and its complementary graph $G'$ (right).

	\vskip.4cm

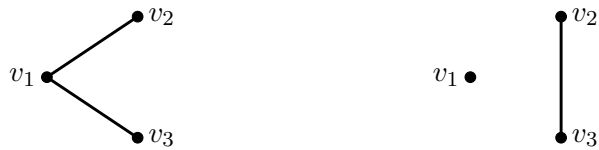
	\begin{center}
		\begin{tikzpicture}[scale=.8]
			\filldraw (3.5,-2) circle (2.5pt);
			\node at (3.1, -2) {$v_1$};
			\filldraw (5,-1) circle (2.5pt);
			\node at (5.4,-1) {$v_2$};
			\filldraw (5,-3) circle (2.5pt);
			\node at (5.4,-3) {$v_3$};
			\draw [line width=1.1pt] (3.5,-2)--(5,-3);
			\draw [line width=1.1pt] (3.5,-2)--(5,-1);

			\filldraw (10.5,-2) circle (2.5pt);
			\node at (10.1, -2) {$v_1$};
			\filldraw (12,-1) circle (2.5pt);
			\node at (12.4,-1) {$v_2$};
			\filldraw (12,-3) circle (2.5pt);
			\node at (12.4,-3) {$v_3$};
			\draw [line width=1.1pt] (12,-1)--(12,-3);
		\end{tikzpicture}
	\end{center}

	In this case the bases obtained for each $V_j$, $j=1,2,3,4$, are
	\begin{eqnarray*}
		& v_1,v_2,v_3 ;\\
		& [v_1,v_2],[v_1,v_3] ;\\
		& [v_1,[v_1,v_2]], [v_1,[v_1,v_3]], [[v_1,v_2],v_2], [[v_1,v_3],v_2], [[v_1,v_3],v_3] ;\\
		& [v_1, [v_1, [v_1, v_2]]], [v_1, [v_1, [v_1, v_3]]], [[[v_1, v_2], v_2], v_2], [[[v_1, v_3], v_2], v_2], \\
		& [[[v_1, v_3], v_3], v_2], [[[v_1, v_3], v_3], v_3], [[v_1, [v_1, v_2]], v_2], [[v_1, [v_1, v_3]], v_2], \\
		& [[v_1, [v_1, v_3]], v_3], [[v_1, v_2], [v_1, v_3]] .
	\end{eqnarray*}

	\begin{remark}
		If $\{v_i,v_j\}$ is an edge and $v_i<v_j$, then $[v_i,[v_i,\dots,[v_i,v_j]\dots]]$, with $k-1$ times $v_i$, is a basis element lying in $V_k$.
	\end{remark}

	Let $m=|V|$ be the number of vertices of $G$.
	Each basis element $\omega$ has a multidegree $d(v)$, the $m$-tuple of multiplicities of each vertex in the underlying word.
	In the example above, $d([[v_1,v_2],v_2])=(1,2,0)$ and $d([[v_1, [v_1, v_3]], v_2])=(2,1,1)$.
	Given two basis elements $\omega_i$ and $\omega_j$, their bracket is
	\[ [\omega_i,\omega_j]=\sum_k c_k \omega_k, \]
	where $c_k\ne 0$ only if
	\begin{equation}\label{eqn:multidegree}
		d(\omega_k)= d(\omega_i)+d(\omega_j).
	\end{equation}

	\section{On the rigidity of nilpotent Lie algebras}

	As a general setting we consider the variety $\L_n$ of complex Lie algebras of dimension $n$,
	the subvariety $\N_n$ of nilpotent Lie algebras, and the subvarieties $\N_{n,k}$ of nilpotent Lie algebras
	which are at most $k$-step nilpotent.
	Note that
	\[ \L_n\supseteq \N_n=\N_{n,n-1}\supseteq \N_{n,n-2}\supseteq\dots\supseteq \N_{n,1}. \]
	
	A $\mu\in\L_n$ is a Lie bracket in $\C^n$ and for $g\in GL(n,\C)$, $g\cdot\mu$ is the Lie bracket
	\[ g\cdot\mu(x,y)=g \mu(g^{-1}x,g^{-1}y). \]
	The orbit of $\mu$ under this action, $\O(\mu)$, is the isomorphism class of $\mu$.
	If $\O(\mu)$ is Zariski open in $\L_n$, $\mu$ is said to be rigid.

	This action preserves the subvarieties $\N_{n,k}$ and a $\mu\in\N_{n,k}$ is said to be \emph{nilpotently rigid}
	if $\O(\mu)$ is open in $\N_n$ and it is said to be \emph{$k$-rigid} if $\O(\mu)$ is open in $\N_{n,k}$.

	\begin{remark}
		On the varieties $\L_n$, $\N_n$ and $\N_{n,k}$, in addition to the Zariski topology, one can also consider the finer Euclidean topology.
		It is well known that an orbit is Euclidean open if and only if it is Zariski open.
	\end{remark}

	Observe that if $\mu$ is not rigid in $\N_{n,k}$, then it is not rigid in $\N_{n,k+i}$ for $i\ge 1$, and also not rigid in $\N_n$ and $\L_n$.

	\bigskip

	We will sometimes write, for convenience, $[x,y]$ instead of $\mu(x,y)$.

	\subsection{Some known results}\label{subsec:known-results}

	The following known results will be referenced in Section 4.

	\begin{enumerate}
		\item {\sl Abelian Lie algebras $\a_n$.} Are not rigid in $\N_{n,k}$ for all $k$, except for $n=1, 2$.
		\item {\sl Dimension $<6$.} Nilpotent Lie algebras of small dimension are classified and it turns out that for $n<6$ there are finitely many isomorphism classes.
		The orbits closures in $\N_n$ are described in \cite{GO1}. It follows that for $n<6$, $\N_n$ has only one rigid algebra.
		Namely, for $n=1,2$ the abelian are rigid in $\N_1$ and $\N_2$ respectively; in $\N_3$ the Heisenbeg Lie algebra $\h_1$ is rigid;
		in $\N_4$ the standard filiform is the only rigid one; and in $\N_5$ there is also only one rigid algebra $\g^6_5$.
		\item  {\sl Graded.} Graded nilpotent Lie algebras as in \eqref{eqn:graded} are never rigid in $\L_n$ \cite{GO2}, but they might be rigid in $\N_{n,k}$ for some $k$.
		\item {\sl Free nilpotent.} The free $k$-step nilpotent Lie algebras are rigid in $\N_{n,k}$.  However they are not rigid in $\N_{n,k+1}$ \cite{BTS}
		(with the only exception of $\h_1$, since $\N_{3,3}=\N_{3,2}$).
		\item {\sl Heisenberg.} The Heisenberg Lie algebras $\h_{m}$ are rigid in $\N_{2m+1,2}$. However they are not rigid in $\N_{2m+1,3}$ \cite{BTS}
		(with the only exception of $\h_1$, since $\N_{3,3}=\N_{3,2}$).
		\item {\sl With an abelian factor.} Nilpotent Lie algebras with an abelian factor are (almost) never rigid. That is, if $\g$ is $k$-step nilpotent
		then $\g\oplus\a_l$ is not $k$-rigid \cite[Theorem 6.10]{BTS}. The only exceptions are $\a_1$, $\a_2$ and $\h_1\oplus\a_1$.
		\item {\sl 2-step graph Lie algebras.} 2-step graph Lie algebras $\g(2,G)$ are never 3-rigid \cite[Theorem 3.7]{BTS} with the only exceptions of $\a_1$, $\a_2$ and $\h_1$.
	\end{enumerate}

	So far as we know, besides the free $k$-step and Heisenberg Lie algebras, there were no other $k$-step nilpotent Lie algebras $k$-rigid of dimension $\ge 6$.

	\subsection{Linear deformations}\label{subsec:linear-def}

	The content of this subsection is from \cite{BTS}.

	Given a Lie algebra $\g$ with bracket $\mu$, if $\sigma$ is an alternating bilinear map from $\g\times\g\to\g$ such that
	\[ \mu_t=\mu+t\sigma \]
	is a Lie bracket for all $t\in\C$, we say that $\mu_t$ is a linear deformation of $\mu$.
	If for all $t\ne 0$ $\mu_t$ is not isomorphic to $\mu$, then $\mu$ is not rigid.

	Given two linearly independent elements $a_1,a_2\in\g$, let $\h$ be any subspace of $\g$ such that
	$\g=\operatorname{span}\{a_1,a_2\}\oplus\h$.
	For any $y\in\g$, we consider the alternating bilinear map $\sigma^{a_1,a_2}_y:\g\times\g\to\g$ defined by
	\begin{gather*}
		\sigma^{a_1,a_2}_y(a_1,a_2)=y, \\
		\sigma^{a_1,a_2}_y(a_1,h)=\sigma^{a_1,a_2}_y(a_2,h)=\sigma^{a_1,a_2}_y(h,h')=0,
	\end{gather*}
	for all $h,h'\in\h$.

	\begin{theorem}[\cite{BTS}]\label{thm:bts}
		Let $\g$ be a nilpotent Lie algebra with bracket $\mu$. Let $\h$ be a codimension 2 subalgebra of $\g$ and let $a_1,a_2$ be such that
		$\g= \langle a_1,a_2\rangle \oplus \h$.
		Then for every $y$ in the centralizer of $\h$, $\mu_t=\mu+t \sigma^{a_1,a_2}_y$ is a linear deformation of $\mu$.
	\end{theorem}

	\subsection{Criteria for 2-step nilpotent Lie algebras}

	Let $\n$ be a $k$-step nilpotent Lie algebra of dimension $n$ with bracket $\mu$.
	The $k$-rigidity of $\mu$ follows from the vanishing of a suitable cohomolgy group, $H^2_{k\text{-nil}}(\mu,\mu)$ \cite{BTS}.

	For $k=2$ it is defined as
	\begin{equation}\label{eqn:H2-nil}
		H^2_{2\text{-nil}}(\mu,\mu)= \frac{\ker(\delta^2)\cap\ker(\eta_2)}{\Im(\delta^1)} = \frac{\ker(\eta_2)}{\Im(\delta^1)},
	\end{equation}
	where $\delta^1: \Lambda^1 ({\K^n}^*) \to \Lambda^2 ({\K^n}^*)$ and  $\delta^2: \Lambda^2 ({\K^n}^*) \to \Lambda^3 ({\K^n}^*)$
	are the differentials of the Chevalley-Eilenberg complex of $\mu$, and $\eta_2: \Lambda^2 ({\K^n}^*) \to \Lambda^3 ({\K^n}^*)$
	is defined by
	\[ \eta_2(\sigma)(x,y,z)=\mu(\sigma(x,y),z)+\sigma(\mu(x,y),z). \]

	\begin{theorem}[\cite{BTS}]
		Let $\mu\in\N_{n,2}$. If $H^2_{2\operatorname{-nil}}(\mu,\mu)=0$, then $\mu$ is $2$-rigid.
	\end{theorem}

	To prove non-2-rigidity for 2-step graph Lie algebras, Theorem \ref{thm:non-rigidity-k2} below suffices.
	To prove it we use a preliminary result that can be found in \cite{A}.

	Recall that if $\mathfrak{n}$ is a 2-step nilpotent Lie algebra, then it admits a decomposition
	$\mathfrak{n} = \mathfrak{v} \oplus \mathfrak{z}$, where $\mathfrak{z} = Z(\mathfrak{n})$ is the center of $\mathfrak{n}$, and $\mathfrak{v}$ is any direct complement of $\mathfrak{z}$.
	If $\n=\g(2,G)$ for a graph $G=(V,E)$ and $\n$ has no abelian factor, then its center is $\z=\langle E \rangle$
	and a direct complement is $\v=\langle V \rangle$, so $\g(2,G)=\langle V \rangle \oplus \langle E \rangle$.

	\begin{lemma}[\cite{A}]
		Let $\n=\v\oplus\z$ be a 2-rigid 2-step nilpotent Lie algebra.
		Then $\Lambda^2\v^*\otimes\z\subseteq \Im(\delta^1)$.
	\end{lemma}

	\begin{proof}
		If $\varphi \in \Lambda^2\v^*\otimes\z$, then
		$\mu_t=\mu+t\varphi$ is a $2$-step analytic deformation and then is equivalent to the trivial deformation of $\mu$ \cite[Theorem 7.1]{Mm}.
		[In characteristic zero, analytic and geometric rigidity are equivalent.]
		This implies that $\varphi\in \Im(\delta^1)$.
	\end{proof}

	\begin{remark} In \cite{A} there is also the proof of the converse of the previous lemma (that is not necessary in this paper).
	\end{remark}

	\begin{theorem}\label{thm:non-rigidity-k2}
		Let $\n=\v\oplus\z$ be a 2-step nilpotent Lie algebra, where $\z=Z(\n)$ is its center.
		If there exist $v,w\in\v$, two linearly independent elements such that $[v,w]=0$ and $\langle [v,\n]\cup [w,\n]\rangle \subsetneq \z$, then $\n$ is not $2$-rigid.
	\end{theorem}

	\begin{proof}
		Choose $z\in \z-\langle [v,\n]\cup [w,\n]\rangle$ and consider
		$\varphi=v^*\wedge w^*\otimes z$.
		On the one hand
		$\varphi\in \Lambda^2\v\otimes \z$.
		On the other hand $\varphi\notin \Im(\delta^1)$.
		In fact, if $\varphi=\delta^1(f)$ for a linear map $f:\n\rightarrow\n$,
		then
		\begin{eqnarray*}
			(v^*\wedge w^*\otimes z)(v,w) &=& \delta^1f(v,w) \\
			z &=& [f(v),w]+[v,f(w)]-f([v,w]) \\
			z &=& [f(v),w]+[v,f(w)]
		\end{eqnarray*}
		and hence $z\in \langle [v,\n]\cup [w,\n]\rangle$, which is a contradiction since we started this proof by choosing $z\in \z-\langle [v,\n]\cup [w,\n]\rangle$.
		
		Therefore $\Lambda^2\v\otimes\z\not\subseteq \Im(\delta^1)$ and because of the previous lemma $\mu$ is not $2$-rigid.  
	\end{proof}

	\section{The rigidity problem for $k$-step nilpotent \\ graph Lie algebras, for $k\ge 3$}

	In this section we prove that for $k\ge 3$, the $k$-step nilpotent graph Lie algebras are not $k$-rigid with the only exceptions of
	those associated with the complete graphs, the free $k$-step nilpotent Lie algebras.

	The following theorem, which provides sufficient conditions to ensure $\g(k,G)$ is not $k$-rigid,
	applies to a wider class of naturally graded nilpotent Lie algebras.
	
	\begin{theorem}\label{thm:graded}
		Let $\g=V_1\oplus\dots\oplus V_k$ be a $k$-step nilpotent naturally graded Lie algebra with $k\geq 3$.
		If there exist two independent commuting elements $a_1, a_2\in V_1$ and $y\in V_k$ such that
		$y\not\in [\g^2,\g^2]+[a_1,\g^2]+[a_2,\g^2]$, then $\g$ is not $k$-rigid.
	\end{theorem}
	
	\begin{proof}
		Let $\h$ be the subalgebra $V_1'\oplus V_2\oplus\dots\oplus V_k$, where $V_1'$ is a direct complement of $\langle a_1,a_2\rangle$ in $V_1$.
		The elements $a_1,a_2$ and $y$ satisfy the conditions of Theorem \ref{thm:bts}.
		So that, if $\mu$ is the bracket of $\g$, we may consider the linear deformation of $\mu$ given by
		\[ \mu_t=\mu+t \sigma^{a_1,a_2}_y. \]
		It is straightforward that $\mu_t$ is $k$-step nilpotent, for all $t\in\C$.
		
		We prove now that $\mu_t$ is not naturally graded  for all $t\neq 0$ and hence $\g$ is not $k$-rigid.
		
		\medskip

		Suppose that $(\g,[\ ,\ ]_t)$ is naturally graded, then there exits a linear decomposition $\g=W_1\oplus\ldots\oplus W_k$ such that  $[W_i,W_j]_t\subseteq W_{i+j}$. Observe the following facts:
		\begin{enumerate}
			\item The descending central series of $[\ ,\ ]$ and $[\ ,\ ]_t$ are the same, then, $V_j\oplus\ldots\oplus V_k=W_j\oplus\ldots\oplus W_k$,  for all $1\leq j\leq k$.
			\item $a_1, a_2\in V_1$, so  $\a_1, a_2\notin[\g,\g]=[\g,\g]_t=W_2\oplus\ldots\oplus W_k$. Then, the decompositions of $a_1$ and $a_2$ in the direct sum $W_1\oplus\ldots\oplus W_k$ are
			$$
			\begin{aligned}
				a_1&=w_1+w_2+\ldots+w_k=w_1+b\\
				a_2&=w_1'+w_2'+\ldots+w_k'=w_1'+c,\\
			\end{aligned}
			$$
			with $w_1, w_1'\in W_1$ linear independent and $b, c \in[\g,\g]_t=[\g,\g]$.
			\item Since $b, c\in V_2\oplus\ldots\oplus V_k$, then $[b,v]_t=[b,v]$ and $[c,v]_t=[c,v]$, for all $v\in\g$.
		\end{enumerate}
		
		Now, bracketing $a_1$ and $a_2$, we have
		$$
		\begin{aligned}
			ty&=[a_1,a_2]_t\\
			&=[w_1+b,w_1'+c]_t\\
			&=\underbrace{[w_1,w_1']_t}_{\in W_2}+\underbrace{[w_1,c]_t+[b,w_1']_t+[b,c]_t}_{\in W_3\oplus\ldots\oplus W_k}\\
		\end{aligned}
		$$
		Since $ty\in V_k=W_k$ and $k\geq3$, then $[w_1,w_1']_t=0$ and
		$$
		\begin{aligned}
			ty&=[w_1,c]_t+[b,w_1']_t+[b,c]_t\\
			&=[a_1-b,c]_t+[b,a_2-c]_t+[b,c]_t\\
			&=[a_1,c]_t-[b,c]_t+[b,a_2]_t-[b,c]_t+[b,c]_t\\
			&=[a_1,c]_t-[b,c]_t+[b,a_2]_t\\
			&=[a_1,c]-[b,c]+[b,a_2]\,.\\
		\end{aligned}
		$$
		From this, we have $ty\in [\g^2,\g^2]+[a_1,\g^2]+[a_2,\g^2]$ and then $t=0$.
	\end{proof}

	From the previous theorem, for the particular case $k=3$, we have the following corollary.

	\begin{corollary}
		Let $\g=V_1\oplus V_2\oplus V_3$ be a $3$-step nilpotent naturally graded Lie algebra, if there exist linear independent elements  $a_1, a_2\in V_1$
		such that $[a_1,a_2]=0$, and  $y\in V_3$ such that $y$ is not an element of the vector space $[a_1,V_2]+[a_2,V_2]$, then $\g$ is not $3$-rigid.
	\end{corollary}

	The main result of this section is now straightforward.

	\begin{theorem}\label{thm:k3}
		Let $G=(V,E)$ be a simple graph with $|V|\ge 3$ which is not a complete graph and let $k\ge 3$.
		Then the $k$-step nilpotent Lie algebra $\g=\g(k,G)$ is not $k$-rigid.
	\end{theorem}

	\begin{proof}
		We may assume that $G$ has at least one edge, otherwise $\g$ is abelian and not $k$-rigid.
		Since $|V|\ge 3$ and $G$ is not the complete graph, there are vertices $v_1,v_2,v_3$ such that $\{v_1,v_2\}$ is an edge and
		$\{v_2,v_3\}$ is not an edge of $G$.

		By declaring $v_1<v_2<v_3$ and considering the basis given in Subsection \ref{subsec:hall}, we have that
		$y=[v_1,[v_1,[\dots,[v_1,v_2]\dots]]]$ with multidegree $d(y)=(k-1,1,0,\dots,0)$ is a basis element in $V_k$.
		Observe that it follows from \eqref{eqn:multidegree} that $y\not\in [\g^2,\g^2]+[v_2,\g^2]+[v_3,\g^2]$.
		Hence, by Theorem \ref{thm:graded}, $\g=\g(k,G)$ is not $k$-rigid.
	\end{proof}

	\section{The rigidity problem for $2$-step nilpotent \\ graph Lie algebras}\label{sec:k2}

	Unlike the previous case for $k>2$, we find that there are some $2$-rigid graph Lie algebras associated with a non-complete graph.
	The exceptions are finite and all of them are associated with graphs that have at most 4 vertices.

	Note that if $m\le 3$ and $G$ is not the complete graph $K_3$, then $\dim\g(2,G)\le 5$.
	So that, all these cases have been discussed already.

	\begin{theorem}\label{thm:k2}
		Let $G=(V,E)$ be a simple graph with $|V|>4$. If it is not a complete graph, then $\g(2,G)$ is not $2$-rigid.
	\end{theorem}

	\begin{proof}
		We may assume that $\n=\g(2,G)$ has no abelian factor (Subsection \ref{subsec:known-results}, (6)), so every vertex belongs to an edge.
		Recall that $\n=\langle V \rangle \oplus \langle E \rangle$, where its center is $\z=\langle E \rangle$. 
		
		Since $G$ is not a complete graph, there exist two vertices $v,w$ such that $\{v,w\}\notin E$.
		Consider the subgraph generated by three distinct vertices $v_1,v_2,v_3$ different from $v$ and $w$ ($|V|\ge 5$).
		\begin{itemize}
			\item[(i)] If it has at least one edge $\{v_i,v_j\}$, then $v$ and $w$ satisfy the hypothesis of Theorem \ref{thm:non-rigidity-k2}.
			In fact, $[v,w]=0$ and $v_i\wedge v_j\in\z-\langle [v,\n] \cup [w,\n] \rangle$.
			
			\item[(ii)] If it has no edges, then $v_1$ and $v_2$ satisfy the hypothesis of Theorem \ref{thm:non-rigidity-k2}.
			Since $v_3$ is not an isolated vertex of $G$, there exists a vertex $z$, different from $v_1$ and $v_2$, such that $\{v_3, z\}\in E$. 
			Then, $[v_1,v_2]=0$ and $v_3\wedge z \in \z- \langle [v_1,\n] \cup [v_2,\n] \rangle$.
		\end{itemize}
	   It follows, then, that in both cases $\n$ is not 2-rigid. 
	   The proof is complete.
	\end{proof}

	It remains only to consider the case with $|V|=4$.
	There are 11 isomorphism classes of graphs with 4 vertices including the complete graph $K_4$.
	Four of the other 10 have an isolated vertex, so the corresponding algebras have an abelian factor and are not $2$-rigid.

	Let us consider the remaining 6 cases divided into two families.

	\vskip.3cm

	\noindent{\sc Family 1.} The 2-step algebras associated with the graphs
	\vskip.3cm
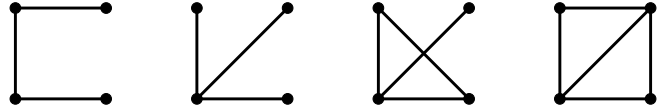
	\begin{center}
		\begin{tikzpicture}[scale=.8]
			\filldraw (3.5,-1.5) circle (2.5pt);
			\filldraw (5,-1.5) circle (2.5pt);
			\filldraw (3.5,-3) circle (2.5pt);
			\filldraw (5,-3) circle (2.5pt);
			\draw [line width=1.1pt] (3.5,-3)--(5,-3);
			\draw [line width=1.1pt] (3.5,-3)--(3.5,-1.5);
			\draw [line width=1.1pt] (3.5,-1.5)--(5,-1.5);

			\filldraw (6.5,-1.5) circle (2.5pt);
			\filldraw (8,-1.5) circle (2.5pt);
			\filldraw (6.5,-3) circle (2.5pt);
			\filldraw (8,-3) circle (2.5pt);
			\draw [line width=1.1pt] (6.5,-3)--(8,-3);
			\draw [line width=1.1pt] (6.5,-3)--(6.5,-1.5);
			\draw [line width=1.1pt] (6.5,-3)--(8,-1.5);
			
			\filldraw (9.5,-1.5) circle (2.5pt);
			\filldraw (11,-1.5) circle (2.5pt);
			\filldraw (9.5,-3) circle (2.5pt);
			\filldraw (11,-3) circle (2.5pt);
			\draw [line width=1.1pt] (9.5,-3)--(11,-3);
			\draw [line width=1.1pt] (9.5,-3)--(9.5,-1.5);
			\draw [line width=1.1pt] (9.5,-1.5)--(11,-3);
			\draw [line width=1.1pt] (9.5,-3)--(11,-1.5);
			
			\filldraw (12.5,-1.5) circle (2.5pt);
			\filldraw (14,-1.5) circle (2.5pt);
			\filldraw (12.5,-3) circle (2.5pt);
			\filldraw (14,-3) circle (2.5pt);
			\draw [line width=1.1pt] (12.5,-3)--(14,-3);
			\draw [line width=1.1pt] (12.5,-3)--(12.5,-1.5);
			\draw [line width=1.1pt] (12.5,-1.5)--(14,-1.5);
			\draw [line width=1.1pt] (14,-1.5)--(14,-3);
			\draw [line width=1.1pt] (12.5,-3)--(14,-1.5);
		\end{tikzpicture}
	\end{center}
	are not 2-rigid.

	In all these graphs $G$ there are vertices $v_1,v_2$ such that $\{v_1,v_2\}$ is not an edge, but $\{v_3,v_4\}$ is an edge.
	Then for $\n=\g(2,G)$, $a_1=v_1$ and $a_2=v_2$ satisfy the hypothesis of Theorem \ref{thm:graded}, where $z=v_3\wedge v_4$ is in the center
	of $\n$ but not in $\langle [v_1,\n] \cup [v_2,\n] \rangle$.

	\vskip.3cm

	\noindent{\sc Family 2.} The 2-step algebras associated with the graphs
	\vskip.3cm
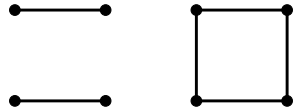
	\begin{center}
		\begin{tikzpicture}[scale=.8]
			\filldraw (3.5,-1.5) circle (2.5pt);
			\filldraw (5,-1.5) circle (2.5pt);
			\filldraw (3.5,-3) circle (2.5pt);
			\filldraw (5,-3) circle (2.5pt);
			\draw [line width=1.1pt] (3.5,-3)--(5,-3);
			\draw [line width=1.1pt] (3.5,-1.5)--(5,-1.5);
			
			\filldraw (6.5,-1.5) circle (2.5pt);
			\filldraw (8,-1.5) circle (2.5pt);
			\filldraw (6.5,-3) circle (2.5pt);
			\filldraw (8,-3) circle (2.5pt);
			\draw [line width=1.1pt] (6.5,-3)--(8,-3);
			\draw [line width=1.1pt] (6.5,-3)--(6.5,-1.5);
			\draw [line width=1.1pt] (6.5,-1.5)--(8,-1.5);
			\draw [line width=1.1pt] (8,-1.5)--(8,-3);
		\end{tikzpicture}
	\end{center}
	are 2-rigid.

	In both cases, their second nil-cohomology vanishes ($H^2_{2-\operatorname{nil}}(\mu,\mu)=0$).
	Since the rigidity of the first one has already been observed in \cite[Proposition 2.4]{A}, we focus on the second one.

	Let's consider the graph that is a square. That is, \(G=(V,E)\), with \(V=\{v_1,v_2,v_3,v_4\}\) and \(E=\{\{v_1,v_2\},\{v_2,v_3\},\{v_3,v_4\},\{v_4,v_1\}\}\).
	Let $\g$ be the 2-step nilpotent Lie algebra associated to $G$ and recall that (see \eqref{eqn:H2-nil})
	\[ H^2_{2-\nil}(\mu,\mu)=\frac{\ker(\eta_2)}{\Im(\delta^1)}. \]
	If \(\sigma\in\ker(\eta_2)\), then $\sigma([v_i,v_j],v_k)+[\sigma(v_i,v_j),v_k]=0$. This implies that \(\sigma(v_1,v_3),\sigma(v_2,v_4)\in Z(\g)\) and
	\[
	\begin{aligned}
		\sigma(v_1,v_3)&=a[v_1,v_2]+b[v_2,v_3]+c[v_3,v_4]+d[v_4,v_1]\\
		\sigma(v_2,v_4)&=e[v_1,v_2]+f[v_2,v_3]+g[v_3,v_4]+h[v_4,v_1],
	\end{aligned}
	\]
	for some constants $a,b,c,d,e,f,g,h\in\C$.
	The linear function $F:\g\rightarrow \g$ defined by
	\[
	\begin{aligned}
		F(v_1)&=bv_2-cv_4\\
		F(v_2)&=-hv_1+gv_3\\
		F(v_3)&=av_2-dv_4\\
		F(v_4)&=-ev_1+fv_3\\
		F([v_1,v_2])&=[F(v_1),v_2]+[v_1,F(v_2)]-\sigma(v_1,v_2)\\
		F([v_2,v_3])&=[F(v_2),v_3]+[v_2,F(v_3)]-\sigma(v_2,v_3)\\
		F([v_3,v_4])&=[F(v_3),v_4]+[v_3,F(v_4)]-\sigma(v_3,v_4)\\
		F([v_4,v_1])&=[F(v_4),v_1]+[v_4,F(v_1)]-\sigma(v_4,v_1)\\
	\end{aligned}
	\]
	satisfies \(\delta^1(F)=\sigma\).
	Therefore, $H^2_{2-\nil}(\mu,\mu)=0$.

	\bigskip

	By combining Theorem \ref{thm:k3}, Theorem \ref{thm:k2} and the discussion in this section, we can now state the following classification result.

	\begin{theorem}
		Let $G=(V,E)$ be a simple graph with $|V|>1$ and let $k\ge 2$.
		The $k$-step nilpotent Lie algebra $\g(k,G)$ is $k$-rigid if and only if $G$ is the complete graph $K_m$, for $m>1$, or $k=2$ and $G$ is one of the following 5 graphs:
		\vskip .6cm
		\begin{center}
			\begin{tikzpicture}[scale=.9]
				\filldraw (-2,0.5) circle (2.5pt);
				\filldraw (-1,0.5) circle (2.5pt);
				
				\filldraw (1.5,1) circle (2.5pt);
				\filldraw (1,0) circle (2.5pt);
				\filldraw (2,0) circle (2.5pt);
				\draw [line width=1.1pt] (1,0)--(2,0);
				
				\filldraw (4.5,1) circle (2.5pt);
				\filldraw (4,0) circle (2.5pt);
				\filldraw (5,0) circle (2.5pt);
				\draw [line width=1.1pt] (4,0)--(5,0);
				\draw [line width=1.1pt] (4,0)--(4.5,1);

				\filldraw (7,1) circle (2.5pt);
				\filldraw (8,1) circle (2.5pt);
				\filldraw (7,0) circle (2.5pt);
				\filldraw (8,0) circle (2.5pt);
				\draw [line width=1.1pt] (7,1)--(8,1);
				\draw [line width=1.1pt] (7,0)--(8,0);
				
				\filldraw (10,1) circle (2.5pt);
				\filldraw (11,1) circle (2.5pt);
				\filldraw (10,0) circle (2.5pt);
				\filldraw (11,0) circle (2.5pt);
				\draw [line width=1.1pt] (10,1)--(11,1);
				\draw [line width=1.1pt] (10,0)--(11,0);
				\draw [line width=1.1pt] (10,1)--(10,0);
				\draw [line width=1.1pt] (11,0)--(11,1);
			\end{tikzpicture}
		\end{center}

		\

	\end{theorem}


\end{document}